\documentclass[a4,11pt]{amsart}
\usepackage{amsmath,amssymb, amscd}

\usepackage[utf8]{inputenc} 

\usepackage[top=3.2cm, left=2.3cm, bottom=2.3cm, right=2.3cm]{geometry}
\usepackage{setspace}

\usepackage{tikz}
\usetikzlibrary{shapes,arrows}
\usepackage[all,ps]{xypic}

\newcommand{\B}[1]{\mathbb{#1}}
\newcommand{\C}[1]{\mathcal{#1}}

\newcommand{\rmod}[1]{\text{\rm {\bf Mod}-}{#1}}
\newcommand{\lmod}[1]{{#1}\text{\rm -{\bf Mod}}}

\newcommand{\gldim}{{\rm gl.dim}}
\newcommand{\prdim}{{\rm pr.dim}}
\newcommand{\fldim}{{\rm fl.dim}}
\newcommand{\wgldim}{{\rm wgl.dim}}

\newcommand{\Hom}{{\rm Hom}}
\newcommand{\End}{{\rm End}}
\newcommand{\Ext}{{\rm Ext}}
\newcommand{\Tor}{{\rm Tor}}

\renewcommand{\lim}{{\rm lim}}

\renewcommand{\leq}{\leqslant}
\renewcommand{\geq}{\geqslant}

\newcommand{\xra}[1]{\xrightarrow{~#1~}}

\newtheorem{theorem}{Theorem}[subsection]
\newtheorem{proposition}[theorem]{Proposition}

\newtheorem{corollary}[theorem]{Corollary}
\newtheorem{lemma}[theorem]{Lemma}

\theoremstyle{definition}
\newtheorem{definition}[theorem]{Definition}
\newtheorem{remark}[theorem]{Remark}
\newtheorem{example}[theorem]{Example}

\author{Atabey Kaygun}
\address{Department of Mathematics and Computer Science, Bahçeşehir University, Beşiktaş Istanbul, TURKEY}
\email{atabey.kaygun@bahcesehir.edu.tr}
\author{Müge Kanuni}
\address{Department of Mathematics, Boğaziçi University, Bebek Istanbul, TURKEY}
\email{muge.kanuni@boun.edu.tr}

\title{Global Dimensions of Some Artinian Algebras}
\begin{document}
\maketitle

\section*{Introduction}

The structure of arbitrary associative commutative unital artinian
algebras is well-known: they are finite products of associative
commutative unital local algebras~\cite[pg.351,
Cor. 23.12]{Lam:AFirstCourse}.  In the semi-simple case, we have the
Artin-Wedderburn Theorem which states that any semi-simple artinian
algebra (which is assumed to be associative and unital but not
necessarily commutative) is a direct product of matrix algebras over
division rings~\cite[pg.35, Par. 3.5]{Lam:AFirstCourse}.  Along these
lines, we observe a simple classification of artinian algebras and
their representations in Proposition~\ref{ClassificationLemma} (hereby
referred as {\em the Classification Lemma}) in terms of a category in
which each object has a local artinian endomorphism algebra.  This
category is constructed using a fixed set of primitive (not
necessarily central) idempotents in the underlying algebra. The
Classification Lemma is a version of Freyd's {\em Representation
  Theorem}~\cite[Sect. 5.3]{Freyd:AbelianCategories}: from an artinian
algebra $A$ we create a category $\C{C}_A$ on finitely many objects,
and then the category of $A$-modules can be realized as a category of
functors which admit $\C{C}_A$ as their domain.  This construction can
also be thought as a higher dimensional analogue of the semi-trivial
extensions of \cite{Palmer:SemitrivialExtensions} for artinian
algebras.

We use the Classification Lemma as a starting point to develop a novel
effective algorithm to determine finiteness of global dimensions of a
large class of artinian algebras.  As a first step, in
Theorem~\ref{SourceSink} (hereby referred as {\em the Source-Sink
  Theorem}) we get explicit lower and upper bounds for the global
dimension of artinian algebras in terms of the global dimensions of a
pair of artinian subalgebras. Then in the rest of
Section~\ref{GlobalDimension} we develop our algorithm.  This
algorithm is based on a directed graph we construct out of a given
artinian algebra and it allows us to reduce the question of finiteness
of the global dimension of a given artinian algebra to finiteness of
the global dimensions of a finite subset of subalgebras. The graph
$\C{G}_A$ we define in Definition~\ref{GraphOfAnArtinianAlgebra} for
an artinian algebra $A$ is {\bf not} the Auslander-Reiten quiver of
$A$ and it appears to be a simplified version of the {\em natural
  quiver} of an artin algebra defined in
\cite{Li:ArtinAlgebrasViaPathAlgebras}.  See
Subsection~\ref{Explanation} and Figure~\ref{Algorithm} in which we
summarize the results and the algorithm we obtain in
Section~\ref{GlobalDimension}.

As an application of our algorithm, we obtain a new result in
Theorem~\ref{QuotientPathAlgebra}: we show that the quotient of a free
path algebra of a cycle-free graph by a nil ideal has finite global
dimension.  Using the same algorithm we were also able to reproduce
two results from the literature: In
Proposition~\ref{FreePathAlgebraIsHereditary} we prove that free path
algebras over cycle-free graphs are hereditary, and in
Proposition~\ref{IncidenceAlgebras} we prove that incidence algebras
have finite global dimension.  See \cite[Section
9]{Mitchell:TheoryOfCategories} and \cite[Proposition
1.4]{AuslanderReitenSmalo:RepresentationTheoryofArtinAlgebras} for the
classical proofs of these Propositions.

At this point, let us make a careful distinction between {\em an
  artinian algebra} (an algebra which satisfies a descending chain
condition on left, right or bilateral ideals) and {\em an artin algebra} (an algebra which is
finitely generated as a bimodule over a unital commutative artinian
center.)  In this article we assume that our algebras are artinian and
we make no assumption on the finiteness of the $k$-dimensions of the
algebras we work with over the base field $k$. Also, it is common to
restrict to the type of a presentation the base algebra has in the
relevant literature, such as assuming that the base algebra is a
quotient of a free path algebra of a quiver by an ideal contained in
the Jacobson radical, or in its square. See for
example~\cite{AnickGreen:QuotientOfPathAlgebras}. But we make no such
restrictive assumptions about presentations of our artinian algebras
except in Theorem~\ref{QuotientPathAlgebra} where we prove finiteness
of global dimensions of a class of artinian algebras which contains
all incidence algebras.

\subsection*{Plan of this article} 
In Section~\ref{Classification} we either cite or prove some technical
results we need in Section~\ref{GlobalDimension} in proving our main
results. Section~\ref{GlobalDimension} contains our algorithm which
calculates upper and lower bounds for global dimension of artinian
algebras, and all other results needed to develop this algorithm. See
Subsection~\ref{Explanation} and Figure~\ref{Algorithm} for an
overview.  In Section~\ref{PathAlgebras} we apply the results we
obtain in Section~\ref{GlobalDimension} to path algebras and incidence
algebras, and make explicit calculations in Subsection~\ref{Examples}.

\subsection*{Standing assumptions and notation}
Through out the article $k$ is a base field.  We do not make any
assumption about the characteristic of $k$.  All unadorned tensor
products are assumed to be over $k$.  We assume $A$ is a unital
associative algebra over $k$ which may or may not be commutative. When
we say $A$ is artinian (resp. noetherian) we mean $A$ satisfies
descending (resp. ascending) chain condition on both left and right
ideals.  In this article we heavily use idempotents (elements which
satisfy $e^2=e$), but we make no assumptions on whether these
idempotents are in the center.  We will use the notation $Ax$, $xA$
and $AxA$ to denote the left, right and two-sided principal ideals
generated by an element $x\in A$, respectively. The multiplicative
group of invertible elements in $A$ is denoted by $A^\times$.  We will
use $J(A)$ to denote the Jacobson Radical of $A$, i.e. the
intersection of all maximal left ideals in $A$.  We note that the
intersection of all maximal right ideals is also $J(A)$.  All
$A$-modules are assumed to be left $A$-modules unless otherwise
explicitly stated.  We will denote the category of left $A$-modules by
$\lmod{A}$ and the category of right $A$-modules by $\rmod{A}$. For a
finite set $U$, we will use $|U|$ to denote the cardinality of
$U$. Finally, in our graph algebras in forming products of edges we
follow the {\em categorical} convention: a product of the form $fg$ in
a path algebra corresponds to compositions of arrows of the form
$\bullet\xleftarrow{\ f\ }\bullet\xleftarrow{\ g\ }\bullet$.

\subsection*{Acknowledgments}
We would like to thank both Birge Huisgen-Zimmermann and Gene Abrams
for pointing out the mistake in the earlier version of
Proposition~\ref{UniqueIdempotents}.

\section{The classification lemma}\label{Classification}

\subsection{Technical lemmata}\label{Lemmata}~

Below we list few lemmata we need, without proofs. Our main references
are \cite{Lam:Lectures, Lam:AFirstCourse}.
\begin{lemma}\label{Unit}
  For every $x\in A$, if $x\in J(A)$ then $1+x\in A^\times$.
\end{lemma}

\begin{lemma}\label{NonUnitIdeal}\cite[Theorem 19.1]{Lam:AFirstCourse}
  $A$ is local if and only if $A\setminus A^\times$ is a two-sided
  ideal.
\end{lemma}

\begin{lemma}[Nakayama Lemma]
  Let $M$ be a finitely generated left $A$-module.  If $J(A) M = M$
  then $M = 0$.
\end{lemma}

\begin{lemma}\label{IdempotentSplitting}
  Let $A$ be an algebra which splits as a direct sum of two (left)
  ideals $A = I\oplus J$.  Then there are two orthogonal idempotents
  such that $1=e+f$.
\end{lemma}


An $A$-module $M$ is called semi-simple if every submodule of $M$ is a
direct summand of $M$.  An algebra $A$ is called semi-simple if $A$
viewed as an $A$-module is semi-simple.

\begin{lemma}\label{SemiSimpleArtinian}
  Assume $A$ is an artinian $k$-algebra. Then $A$ is left semi-simple
  if and only if it is right semi-simple if and only if $J(A)=0$.
\end{lemma}

\begin{lemma}\label{IdempotentLifting} 
   Let $J$ be a nil ideal.  Then any idempotent in $A/J$ can be lifted
   to an idempotent in $A$.  In other words, for every idempotent
   $\epsilon\in A/J$ there exists an idempotent $e\in A$ such that
   $e+J=\epsilon$.
 \end{lemma}



\begin{definition}
  An idempotent $e\in A$ is called primitive when for every pair of
  orthogonal idempotents $e_1, e_2\in A$ if $e = e_1 + e_2$ then $e_1
  = 0$ or $e_2=0$.
\end{definition}

\begin{lemma}\label{PrimitiveIdempotentLifting}\cite[Proposition 21.22]{Lam:AFirstCourse} 
  Assume $A$ is an artinian algebra, $e \in A$ idempotent, $J \subset
  J(A)$. Then $e+J$ is a primitive idempotent of $A/J$ if and only if
  $e$ is a primitive idempotent of $A$.
\end{lemma}

\begin{lemma}\label{OnlyIdempotents}\cite[Corollary 19.19]{Lam:AFirstCourse}
  Assume $A$ is an artinian $k$-algebra.  Then $A$ is local if and
  only if the only idempotents of $A$ are $1$ and $0$.
\end{lemma}

\subsection{Idempotents}~

\begin{proposition}\label{LocalDecomposition}
  Assume $A$ is an artinian $k$-algebra and $e\in A$ is an idempotent.
  Then $e\in A$ is a primitive idempotent if and only if $eAe$ is a
  local ring.  
\end{proposition}

\begin{proof}
  Assume $e$ is a primitive idempotent.  Note that, regardless of $e$
  being primitive, $eAe$ is a $k$-algebra with $e\in eAe$ being its
  identity element. Assume $eAe$ is not local.  Since $eAe$ is also an
  artinian algebra on its own right, there exists an idempotent
  $efe\in eAe$ which is not 0 or $e$ by Lemma~\ref{OnlyIdempotents}.
  Then $e = efe + (e - efe)$ and observe that since $efe(e-efe) =
  (e-efe)efe = 0$, we write $e$ as a sum of two orthogonal idempotents
  in $A$ which is a contradiction since $e$ is primitive.  So, $e$ is
  primitive implies $eAe$ is a local artinian $k$-algebra.
  Conversely, assume $eAe$ is local artinian.  Then we know that the
  only idempotents of $eAe$ are $e$ and $0$.  Assume that $e$ can be
  written as a sum of two orthogonal idempotents in $A$ as $e=u+v$.
  Then we see that $eu = ue = u^2 =u$ and $ev = ve = v^2 = v$ and
  therefore $e = u + v = eue + eve$ in $eAe$.  Since the only
  idempotents in $eAe$ are $0$ and $1$ we see that $eue = u = 0$ or
  $eve = v = 0$, i.e. that $e$ is primitive.  Thus we conclude that if
  $eAe$ is local artinian then $e$ is primitive.
\end{proof}

\begin{proposition}\label{SplittingIdempotents}
  (Compare this with \cite[Theorem 23.6]{Lam:AFirstCourse}. Note that
  any artinian ring is semi-perfect.)  Let $A$ be an artinian
  $k$-algebra.  Then the unit $1\in A$ is a finite sum of pairwise
  orthogonal primitive idempotents.
\end{proposition}

\begin{proof} 

  Since $A$ is artinian, so is $A':=A/J(A)$ and since the Jacobson
  radical of $A'$ is trivial we see that $A'$ is (left) semi-simple by
  Lemma~\ref{SemiSimpleArtinian}.  So, $A'$ splits as finite direct
  sum of minimal (left) ideals $\bigoplus_{i=1}^n I_i$. Then by
  Lemma~\ref{IdempotentSplitting}
  there exists pairwise orthogonal idempotents $e_i'\in A'$ such that
  $I_i = A'e_i'$ for $i = 1,\ldots,n$, and $1_{A'} = \sum_{i=1}^n
  e_i'$.  Now to show that $e_i'$ is primitive, assume $e_i' = f_1 +
  f_2$ for some orthogonal idempotents $f_1,f_2$ of $A'$.  Clearly,
  $A'e_i'\subset A'f_1 \oplus A'f_2$. For the converse inclusion,
  observe that any $x = r_1f_1+r_2f_2 \in A'f_1 \oplus A'f_2$, $xe_i'
  =x \in A'e_i'$.  Then $A'e_i'= A'f_1 \oplus A'f_2$. Since, $A'e_i'$
  is a minimal left ideal, we get $A'f_1 = 0$ or $A'f_2 = 0$. So
  $f_1=0$ or $f_2=0$. Any idempotent in $A'$ can now be lifted to an
  idempotent in $A$ by Lemma~\ref{IdempotentLifting}, and by
  Lemma~\ref{PrimitiveIdempotentLifting} the lifts can be chosen from
  primitive idempotents in $A$.
\end{proof}

\begin{proposition}\label{UniqueIdempotents}(Compare this with
\cite[Exercise 21.17]{Lam:AFirstCourse}) Assume $A$ is an artinian
  $k$-algebra.  If $E$ and $F$ are two finite sets of pairwise
  orthogonal primitive idempotents such that $1 = \sum_{e\in E} e =
  \sum_{f\in F} f$ then $|E|=|F|$. 
\end{proposition}

\begin{proof}
  We consider the right ideal $fA$ of $A$ for a fixed $f\in F$.  One
  can see that $\End_A(fA)$ the $k$-algebra of right $A$-module
  endomorphisms of $fA$ is a quotient of the local artinian
  $k$-algebra $fAf$ which acts by multiplication on the left, thus
  itself is local artinian.  We will denote $L_x$ the endomorphism of
  the right $A$-module $fA$ given by any $x\in fAf$.  Then the
  identity morphism $L_f$ is split as $L_f= \sum_{e\in E} L_{fef}$
  since $f=\sum_{e\in E} fef$.  Now, there must be at least one $e\in
  E$ such that $L_{fef}\in (fAf)^\times$ otherwise the element
  $L_{fef}$ would have been in the unique maximal ideal of $fAf$ for
  every $e\in E$ and $L_f = \sum_{e\in E} L_{fef}$ would have not been
  invertible.  Consider the sequence of morphisms of right $A$-modules
  $fA\xra{L_e}eA\xra{L_f}fA$.  Since $L_{fe} = L_{fef}$ is invertible,
  we see that $fA$ must be a direct summand of $eA$ as a right
  $A$-module.  However, $eAe$ is also local artinian, and therefore
  has no idempotents other than $0$ and $1$ by
  Lemma~\ref{OnlyIdempotents}.  This means $eA$ is indecomposable.
  Then $eA$ and $fA$ must be isomorphic as right $A$-modules. Thus get
  that $L_e$ has a two-sided inverse $L_{fx}\colon eA\to fA$ for some
  $fx$ which satisfies $f = fxe$ and $e = efx$.  Then we get $fe =
  fxe^2 = fxe = f$ and $ef = efxe = e^2 = e$.  These equalities imply
  that for each $e\in E$ the idempotent $f\in F$ is unique and vice
  versa.  This gives us $|E|=|F|$.
\end{proof}

\subsection{The classification lemma}

\begin{definition}\label{Category}
  For an artinian $k$-algebra $A$ we associate a canonical $k$-linear
  category $\C{C}_A$ as follows.  The set of objects of $\C{C}_A$ is a
  finite set $E$ of orthogonal primitive idempotents which split $1_A$
  as $1_A = \sum_{e\in E} e$.  This set is nonempty by
  Proposition~\ref{SplittingIdempotents}.  For any $e,f\in
  Ob(\C{C}_A)$ we let
  \[ \Hom_{\C{C}_A}(e,f) = fAe \] and the composition is defined by
  the multiplication operation in $A$. 
\end{definition}

\begin{proposition}[The Classification Lemma]\label{ClassificationLemma}
  Let $\Hom_k(\C{C}_A,\lmod{k})$ be the category of $k$-linear
  functors from $\C{C}_A$ to $\lmod{k}$ and their natural
  transformations.  Then the category $\Hom_k(\C{C}_A,\lmod{k})$ is
  isomorphic to the category $\lmod{A}$ of left $A$-modules.
\end{proposition}

\begin{proof}
  Using decomposition of the identity element $1 = \sum_{e\in E} e$ in
  terms of a set of primitive idempotents, we split any left
  $A$-module as a direct sum of $k$-modules $M = \bigoplus_{e\in E}
  eM$ and for any $x\in fAe$, the left action by $x$ defines a
  $k$-linear morphism $L_x\colon eM\to fM$ for any $e,f\in E$.  In
  short, every left $A$-module defines a functor $\Phi(M)\colon
  \C{C}_A\to \lmod{k}$ which is defined by $\Phi(M)(e) = eM$ on the
  level of objects, and for any $fxe\in \Hom_{\C{C}_A}(e,f)$ we let
  $\Phi(M)(fxe):= L_{fxe}$ the $k$-linear operator defined on $M$ by
  left action of $fxe\in fAe$.  Conversely, for every functor $M\colon
  \C{C}_A\to \lmod{k}$ we define a left $A$-module $\Psi(M):=
  \bigoplus_{e\in E} M(e)$ where the left action of $A$ on $\Psi(M)$
  is defined by using the decomposition $A = \bigoplus_{e,f\in E}
  fAe$, i.e. $fxe\cdot m = 0$ unless $m\in M(e)$ and then $fxe\cdot m$
  is defined as $M(fxe)(m)$, the evaluation of $m$ under the
  $k$-linear morphism $M(fxe)\colon M(e)\to M(f)$.  One can easily see
  that $\Psi\Phi$ is the identity functor on the category of
  $\lmod{A}$, and conversely $\Phi\Psi$ is the identity functor on
  $\Hom_k(\C{C}_A,\lmod{k})$.
\end{proof}

\begin{corollary}
  The category $\C{C}_A$ we defined in Definition~\ref{Category} for a
  given artinian algebra $A$ is independent (up to an isomorphism) of
  the set of primitive idempotents splitting the identity.
\end{corollary}

\begin{proof}
  The categories $\C{C}_A$ one can define for $E$ and for $F$ are both
  isomorphic to the category of left $A$-modules, and therefore,
  isomorphic to each other.
\end{proof}

\begin{remark}
  One can view our Classification Lemma~\ref{ClassificationLemma} as a
  generalization of {\em semi-trivial extensions} developed in
  \cite{Palmer:SemitrivialExtensions} for artinian algebras. The
  Classification Lemma, in fact, allows one to reduce an artinian
  algebra to a $n\times n$ generalized matrix ring where $n=|E|$ is
  the number of primitive idempotents splitting the identity.
  Semi-trivial extensions are the $2\times 2$ examples of this
  construction (not just for artinian algebras) in which one does not
  require primitivity of the idempotents. Since we work with artinian
  algebras, and we reduce our algebras to the submodules $eAf$ using
  idempotents $e$ and $f$ until they can not be further reduced,
  (i.e. until idempotents are primitive) we end up with a larger
  matrix ring. 
\end{remark}

\section{Global dimension of artinian algebras}\label{GlobalDimension}

\subsection{Reductions}

\begin{remark}
  Assume $M$ is a module in $\rmod{A}$.  Define the flat dimension 
  (also known as the weak dimension) of $M$ as
  \[ \fldim_A(M) = \sup \{n\in\B{N}|\ \Tor^A_n(M,N)\neq 0,\ N\in
  Ob(\lmod{A})\} \] where the quantity of the right side is defined to
  be $\infty$ if the subset is unbounded in $\B{N}$.  Similarly we define
  the projective dimension of $M$ as
  \[ \prdim_A(M) = \sup \{n\in\B{N}|\ \Ext_A^n(M,N)\neq 0,\ N\in
  Ob(\rmod{A})\} \] Now we define the global and the weak global
  dimension of $A$ as the quantities
  \[ \gldim(A) = \sup \{\prdim_A(M)\in\B{N}\cup\{\infty\}|\ M\in
  Ob(\rmod{A})\} \] and
  \[ \wgldim(A) = \sup \{\fldim_A(M)\in\B{N}\cup\{\infty\}|\ M\in
  Ob(\rmod{A})\} \] Since $A$ is not necessarily commutative,
  technically we should define a left and a right version of these
  dimensions.  But because of Hopkins-Levitzki Theorem, if $A$ is
  artinian it is also noetherian, and therefore, the left and right
  global dimensions and the global weak dimensions agree.
  (cf. \cite[Cor. 5.60]{Lam:Lectures})
\end{remark}

\begin{definition}
  The two sided bar complex of an algebra $A$ is defined as the
  differential graded $A$-bimodule $CB_*(A):= \bigoplus_{n\geq 0}
  A^{\otimes n+2}$ together with the differentials
  \[ d_n(a_0\otimes\cdots\otimes a_{n+1}) = \sum_{i=0} (-1)^i
  (a_0\otimes\cdots\otimes a_{i}a_{i+1}\otimes \cdots\otimes a_{n+1})
  \]
  We also define another differential graded module $CB'_*(A)$ which
  is defined as a graded module as $\bigoplus_{n\geq 0} A^{\otimes
    n+1}$ with similar differentials.
\end{definition}

\begin{remark}
  The two sided bar complex $CB_*(A)$ gives us a projective resolution
  of $A$ viewed as an $A$-bimodule.  Thus the modified bar complex
  $CB'_*(A)$ is acyclic.  For a right $A$-module $X$ and a left
  $A$-module $Y$, the differential graded $k$-module
  \[ C_*(X;A;Y) := X\otimes_A CB_*(A)\otimes_A Y \] together with the
  induced differentials give us the Tor-groups $\Tor^A_*(X,Y)$ in
  homology.  We also define a second complex
  \[ C'_*(X;A;Y) := X\otimes_A CB'_*(A)\otimes_A Y \] together with
  the induced differentials.  Since $C_*(X;A;A)$ is a projective
  resolution of $X$, we see that $C'_*(X;A;A)$ is acyclic.  The same
  is true for $C'_*(A;A;Y)$.
\end{remark}

\begin{proposition}\label{Reduction}
  Consider the subcomplex 
  \[ S_*(X;A;Y):= \bigoplus_{n\in\B{N}}\bigoplus_{e_i\in E} X
  e_0\otimes e_0Ae_1\otimes\cdots\otimes e_{n-1}Ae_n\otimes e_n Y
  \] of the complex $C_*(X;A;Y)$.  Then the injection $S_*(X;A;Y)\to
  C_*(X;A;Y)$ is a quasi-isomorphism.
\end{proposition}

\begin{proof}
  Using the Pierce decomposition of $A$ as $\bigoplus_{e,f\in E}\ eAf$ we can
  split $C_*(X;A;Y)$ as
  \[ \bigoplus_{p\geq 0}\ \ \bigoplus_{e_i\in E} Xe_0\otimes
  e_1Ae_2\otimes\cdots\otimes e_{2p-1}Ae_{2p} \otimes e_{2p+1} Y
  \]
  One can easily see that
  \[ C_*(X;A;Y) = S_*(X;A;Y)\oplus \bigoplus_{e\neq f}
  C'_*(X;A;Ae)\otimes C'_*(fA;A;Y) \] by counting the number of
  idempotents $e_{2i}\neq e_{2i+1}$ in the full decomposition of
  $C_*(X;A;Y)$. The result follows from the fact that $C'_*(X;A;Ae)$
  and $C'_*(fA;A;Y)$ are acyclic.
\end{proof}

\subsection{The directed graph of an artinian algebra}

\begin{definition}\label{GraphOfAnArtinianAlgebra}
  Let us define a directed graph $\C{G}_A$ using an artinian algebra
  $A$ and a complete set of primitive idempotents $E$ splitting the
  identity. The set of vertices of $\C{G}_A$ is $E$. Two idempotents
  $e$ and $f$ are connected with a directed edge $f\leftarrow e$ if
  and only if the $k$-vector subspace $fAe$ is a non-zero.
\end{definition}

\begin{lemma}\label{UniqueGraph}
  $\C{G}_A$ is independent of the choice of any complete set of
  primitive idempotents splitting identity.
\end{lemma}

\begin{proof}
  Assume $E$ and $F$ are two sets of primitive idempotents splitting
  the identity, and let $\C{G}_A^E$ and $\C{G}_A^F$ be the directed
  graphs defined in Definition~\ref{GraphOfAnArtinianAlgebra} for
  these sets.  By Proposition~\ref{UniqueIdempotents}, we have a bijection
  $\omega\colon E\to F$ which satisfies the identities
  \[ e\cdot\omega(e) = e \quad\text{ and }\quad \omega(e)\cdot e =
  \omega(e) \] for every $e\in E$.  These identities also give us an
  isomorphism between the right $A$-modules $eA$ and $\omega(e)A$.
  Now, for every $e,e'\in E$ we have an edge from $e$ to $e'$ if and
  only if $e'Ae$ is non-zero.  The $k$-vector subspace $e'Ae$ is also
  $\Hom_A(eA,e'A)$ the $k$-vector space of $A$-module morphisms from
  $eA$ to $e'A$.  Thus we obtain that $\Hom_A(eA,e'A)$ is non-zero if
  and only if $\Hom_A(\omega(e)A,\omega(e')A)$ is non-zero.  This
  proves that the graphs $\C{G}_A^E$ and $\C{G}_A^F$ we write using
  $E$ and $F$ are isomorphic.
\end{proof}

\begin{remark}
  For the rest of the article, we will fix an artinian algebra $A$ and
  a set $E$ of primitive idempotents splitting the identity $1_A$ of
  $A$.
\end{remark}

Let $\pi_0(\C{G}_A)$ be the set of connected components of
$\C{G}_A$. Note that since $E$ is finite, $\pi_0(\C{G}_A)$ is also
finite. For every $\alpha\in \pi_0(\C{G}_A)$ we let $\chi_\alpha := \{
e\in E|\ \text{the vertex $e$ appears in }\alpha \}$. Note that for
each $\alpha\in\pi_0(\C{G}_A)$
\[ e_\alpha := \sum_{e\in \chi_\alpha} e \] is a {\em central}
idempotent in $A$ and also is the identity element of the subalgebra
$A(\chi_\alpha)$. Moreover, the set $\{e_\alpha|\
\alpha\in\pi_0(\C{G}_A)\}$ forms a set of pairwise orthogonal {\em
  central} idempotents.

\begin{definition}
  For a right $A$-module $X$ we define the {\em support of $X$} as
  \[ E_X := \{e\in E|\ Xe\neq 0\} \] The support of a left $A$-module
  is defined similarly.
\end{definition}

\begin{lemma}~\label{NoPathsNoCohomology} Assume $X$ and $Y$ are right
  $A$-modules. If there are no paths starting in $E_Y$ and ending in
  $E_X$ in the directed graph $\C{G}_A$ then $\Ext_A^n(X,Y)=0$ for
  every $n\geq 1$. The analogous result for $\Tor^A_n(X,N)$ is true
  for every left $A$-module $N$.
\end{lemma}

\begin{proof}
  The proof for the analogous result for $\Tor^A_n(X,N)$ follows
  immediately from Proposition~\ref{Reduction}. Here we give a proof
  for $\Ext_A^n(X,Y)$.  Use the projective resolution of $X$ given by
  $S_*(X;A;A)$. Then the homology of the complex
  \[ \bigoplus_{n\geq 0}\bigoplus_{e_i\in E} \Hom_A(Xe_0\otimes
  e_0Ae_1\otimes\cdots e_{n-1}Ae_n\otimes e_nA,Y) \] with the induced
  differential, calculates $\Ext_A^*(X,Y)$.  On the other hand, we
  have $\Hom_A(Ze\otimes eA, T) \cong \Hom_k(Ze,Te)$ for every $e\in
  E$ and for all right $A$-modules $Z$ and $T$. So, we can rewrite the
  complex above as
  \[ \bigoplus_{n\geq 0}\bigoplus_{e_i\in E} \Hom_k(Xe_0\otimes
  e_0Ae_1\otimes\cdots e_{n-1}Ae_n,Ye_n) \] this time with the
  differentials
  \[ (d\phi)(a_0\otimes a_1\otimes\cdots\otimes a_n) =
  \sum_{i=0}^{n-1}(-1)^i\phi(\cdots\otimes a_ia_{i+1}\otimes\cdots) +
  (-1)^n\phi(a_0\otimes\cdots\otimes a_{n-1})a_n
  \] 
  for every $\phi\in \bigoplus_{e_i\in E} \Hom_k(Xe_0\otimes e_0
  Ae_1\otimes\cdots e_{n-1}Ae_n,Ye_n)$ and homogeneous tensor
  $a_0\otimes a_1\otimes\cdots\otimes a_n\in Xe_0\otimes e_0
  Ae_1\otimes\cdots e_{n-1}Ae_n$. The result follows.
\end{proof}

\begin{proposition}\label{ConnectedComponents}
  $\gldim(A) = \max_{\alpha\in\pi_0(\C{G}_A)}\gldim(A(\chi_\alpha))$
\end{proposition}

\begin{proof}
  Either we use Lemma~\ref{NoPathsNoCohomology}, or we observe that one
  can see that the differential graded $k$-module $S_*(X;A;Y)$ splits
  as a direct sum of differential graded $k$-modules as
  \[ S_*(X;A;Y) = \bigoplus_{\alpha\in \pi_0(\C{G}_A)}
  S_*(X;A(\chi_\alpha);Y) \] and then use the fact that weak global
  dimension and global dimensions agree when $A$ is artinian.
\end{proof}

\begin{remark}
  We will call an artinian algebra $A$ connected if $\C{G}_A$ is
  connected, i.e. has only one connected component.  An immediate
  corollary of Proposition~\ref{ConnectedComponents} is that the
  problem of calculating the global dimension of an artinian algebra
  reduces to calculating the global dimensions of the connected
  subalgebras constructed out of connected components of
  $\C{G}_A$. Therefore, we can assume, without any loss of generality,
  that $A$ is connected.
\end{remark}

\subsection{The Source-Sink Theorem}

\begin{definition}
  For any $U,V\subseteq E$ nonempty subset, we define
  \[ UAV := \bigoplus_{u\in U}\bigoplus_{v\in V} uAv \] and we also
  define a new (non-unital) subalgebra $A(U)$ of $A$ by
  \[ A(U) := UAU = \bigoplus_{u,u'\in U} uAu' \]
\end{definition}

\begin{definition}
  Assume $\Gamma$ is a directed graph with no loops or multiple edges,
  and $U$ is an arbitrary subset of vertices. The quotient graph
  $\Gamma/U$ is the new graph obtained from $\Gamma$ by contracting
  every vertex in $U$ to a single vertex, which is again denoted by
  $U$, and then deleting all loops and oriented multiple edges. A
  vertex is $v$ in $\Gamma$ is called {\em a sink} (resp. {\em a
    source}) if there are no directed edges leaving $v$ (resp. ending
  at $v$.)
\end{definition}

\begin{proposition}\label{LowerBound}
  Assume that there exists a proper subset of primitive idempotents
  $U\subset E$ with the property that the collapsed vertex $U$ in the
  quotient graph $\C{G}_A/U$ is a source or a sink. Then we have \[
  \max\{\gldim(A(U)), \gldim(A(U^c))\} \leq \gldim(A) \] In
  particular, if $A(U)$ has infinite global dimension then so does
  $A$.
\end{proposition}

\begin{proof}
  We will use the fact that $A$ is artinian, and therefore, the global
  and weak dimensions in both left and right variations all agree.
  Assume without generality that $U$ is a sink in $\C{G}/U$, and that
  $M$ is a left $A$-module. One can easily see that $A$ splits as
  \[ A = UAU \oplus UAU^c \oplus U^cAU \oplus U^cAU^c \] However, we
  assumed that $U$ is a sink in $\C{G}/U$. This means $U^cAU$ is
  $0$. Therefore, given any free $A$-module $F= A^{\oplus I}$ one has
  a splitting of the form
  \[ F = A^{\oplus I} = (UAU)^{\oplus I} \oplus (U^c A)^{\oplus I} \]
  Hence any free $A$-module, restricted to a left $A(U)$-module
  becomes a flat left $A(U)$-module because $U^cA$ is a flat left
  $A(U)$-module. The same also holds for projective $A$-modules. This
  means any $A$-projective resolution $P_*$ when restricted to a
  $A(U)$-module becomes $U P_*$ which is a flat
  $A(U)$-module. Therefore for every left $A$-module $M$ we get
  \[ \prdim_A(M)\geq \fldim_{A(U)}(UM) \]
  which implies
  \[ \gldim(A) \geq \wgldim(A(U)) = \gldim(A(U)) \] If we repeat the
  same proof for right modules, we get the same inequality for
  $\gldim(A(U^c))$.
\end{proof}

\begin{theorem}[The Source-Sink Theorem]~\label{SourceSink} Assume $A$
  is an artinian algebra such that there is a proper subset of
  primitive idempotents $U\subset E$ with the property that the
  collapsed vertex $U$ in the quotient graph $\C{G}_A/U$ is a source
  or a sink. Then
  \[ \max\{\gldim(A(U)), \gldim(A(U^c))\} \leq \gldim(A) \leq 1 +
  \gldim(A(U)) + \gldim(A(U^c)) \] Moreover, if $A$ is flat over
  $A(U)$ and $A(U^c)$ then
  \[ \max\{\gldim(A(U)), \gldim(A(U^c))\} \leq \gldim(A) \leq
  \max\{1,\gldim(A(U)),\gldim(A(U^c))\} \]
\end{theorem}

\begin{proof}
  We proved that the lower bounds hold in
  Proposition~\ref{LowerBound}.  Now, without loss of generality we
  assume $U$ is a sink in $\C{G}_A/U$. Otherwise we switch $U$ and
  $U^c$, because $U$ is a sink in $\C{G}_A/U$ if and only if $U^c$ is
  a source in $\C{G}_A/U^c$. We define a filtration
  \[ L^p_n = \sum_{q\leq p}\bigoplus_{n_0+\cdots+n_q=n-q}\
  \bigoplus_{\beta_i\in\{U,U^c\}} X \beta_0 \otimes
  A(\beta_0)^{\otimes n_0} \otimes \beta_0 A \beta_1\otimes\cdots
  \otimes A(\beta_{q-1})^{\otimes n_{q-1}}\otimes
  \beta_{q-1}A\beta_q\otimes A(\beta_q)^{\otimes n_q}\otimes \beta_q Y
  \]
  We consider the sequence of idempotents in each component
  $Xe_0\otimes e_0Ae_1\otimes\cdots e_{n-1}Ae_n\otimes e_nY$ as an
  oriented path in $\C{G}_A$.  The filtration counts the number of
  times each path enters in an out of $U$ and $U^c$. When we consider
  the spectral sequence associated with the filtration we see
  that $E^0_{p,q} = L^p_{p+q}/L^{p-1}_{p+q}$ is given by
  \begin{equation*}
    \bigoplus_{q_0+\cdots+q_p=q}\bigoplus_{\beta_i\neq
      \beta_{i+1}} X \beta_0 \otimes A(\beta_0)^{\otimes q_0} \otimes
    \beta_0 A \beta_1\otimes\cdots \otimes A(\beta_{p-1})^{\otimes
      q_{p-1}}\otimes \beta_{p-1}A\beta_p\otimes A(\beta_p)^{\otimes
      q_p}\otimes \beta_p Y
  \end{equation*}
  Consider the subalgebra $B = A(U)\oplus A(U^c)$. Then $E^0_{p,q}$
  represents a (multi-)product in $\C{D}^+(B)$ the derived category of
  $B$-modules
  \begin{equation}\label{NoAlternatingPaths}
    \bigoplus_{\beta_i\neq \beta_{i+1}} X\beta_0\otimes^R_B \beta_0
    A\beta_1\otimes^R_B\cdots\otimes^R_B\beta_{p-1}A\beta_p\otimes^R_B\beta_pY
  \end{equation}
  Notice that if $A$ is flat over $A(U)$ and $A(U^c)$ then the columns
  in the $E^1$-page collapse at $q=0$ line for $p>0$.  In any case, if
  $U$ were not a source and a sink, we would have had paths where the
  idempotents on the $q=0$ line which would alternate between
  idempotents in $U$ and $U^c$.  The first column $E^0_{0,q}$,
  regardless of $A$ is flat over $A(U)$ and $A(U^c)$, is represented
  by the product
  \[ X U\otimes^R_{B} UY \oplus X U^c\otimes^R_{B} U^c Y \] in the
  derived category while the second column $E^0_{1,q}$ is given by
  \[ X U\otimes^R_{B} UAU^c \otimes^R_{B} U^c Y \] and the rest of the
  terms are zero since $U$ is a sink. In the first page $E^1_{*,*}$
  the height of the first column (after taking supremum over all $X$
  and $Y$) is bounded by the maximum of $\gldim(A(U))$ and
  $\gldim(A(U^c))$. The height of the second column $E^1_{1,q}$ is
  bounded above by $\gldim(A(U))+\gldim(A(U^c))$. Since this is the
  second column, its contribution to the homological dimension of $A$
  is shifted by $1$. In the flat case, we see that $E^1_{p,q}=0$ for
  $p>1$ and $q>1$. The only non-zero term on $q=0$ line is at
  $E^1_{1,0}= \sum_{u\in U}\sum_{v\in U^c}
  Xu\otimes_{A(U)}uAv\otimes_{A(U^c)} vY$ and we have $E^1_{p,0}=0$
  for $p>1$.
\end{proof}

\begin{remark}
  One can obtain a version of Proposition~\ref{LowerBound} using
  Corollary~4.3 of \cite{FossumGriffithReiten:TrivialExtensions}, and
  a version of Theorem~\ref{SourceSink} as a consequence of
  Corollary~3 of \cite{PalmerRoos:GlobalDimension} for artinian
  algebras. 
\end{remark}

\begin{remark}\label{Recursion}
  One can use Theorem~\ref{SourceSink} recursively and reduce the
  calculation of lower and upper bounds for the global dimension of
  $A$ into calculations of the global dimensions of certain
  subalgebras $A(U)$ and $A(U^c)$ provided that the distinguished
  vertex $U$ in the quotient graph $\C{G}_A/U$ is a source or a sink
  at each recursion step. The recursion tree will split the graph
  $\C{G}_A$ into full subgraphs and corresponding artinian
  subalgebras. Since $E$ is finite, the recursion terminates and we
  obtain a partition $U_1,\ldots,U_m$ of $E$ where every subset
  $U'\subset U_i$ is neither source nor a sink in $\C{G}_{A(U_i)}/U'$
  for each $i=1,\ldots,m$. The partition $U_1,\ldots,U_m$, or the
  number of elements in the partition, need not be unique but for any
  such sequence we have the inequality
  \[ \max\{\gldim(A(U_1)),\ldots,\gldim(A(U_m))\}\leq \gldim(A)\leq
  (m-1)+\gldim(A(U_1))+\cdots+\gldim(A(U_m)) \] From this inequality
  we deduce that $A$ has finite global dimension if and only if all of
  these {\em terminal} subalgebras have finite global dimension. We
  will further refine this algorithm in the next section. In the
  current version and in the best possible case we obtain the
  following Corollary.
\end{remark}

\begin{corollary}\label{NoCycle}
  Assume $\C{G}_A$ has no oriented cycles. Then
  \[ \max\{\gldim(eAe)|\ e\in E\}\leq \gldim(A) \leq |E|
  -1 + \sum_{e\in E} \gldim(eAe) \] If $A$ is flat over $A(U)$ for
  every non-empty $U\subset E$ then
  \[ \max\{\gldim(eAe)|\ e\in E\}\leq \gldim(A) \leq
  \max\left(\{1\}\cup\{\gldim(eAe)|\ e\in E\}\right) \]
\end{corollary}


\subsection{Further reductions}

\begin{definition}\label{MoritaComplex}
  Fix a non-empty proper subset $U\subsetneq E$ and define a new
  differential graded $k$-module $R^U_*$ by letting $R^U_0 = A$ and
  $R^U_1 = \bigoplus_{u\in U} Au\otimes uA$ and then
  \[ R^U_n = \bigoplus Au_1\otimes u_1Au_2\otimes\cdots\otimes
  u_{n-1}Au_n\otimes u_n A
  \]
  Notice that $R^U_*$ is a differential graded submodule of
  $CB'_*(A)$. We define $R^U_*(X)$ and $R^U_*(Y)$ similarly for a
  right $A$-module $X$ and a left $A$-module $Y$. Now, let us compute
  the homology $H_n(R^U_*)$. If we consider the brutal truncation
  $R^U_{*>0}$
  \[ \bigoplus_{u\in U} Au \otimes uA \xleftarrow{d_2}
     \bigoplus_{u_1,u_2\in U} Au_1\otimes u_1Au_2\otimes u_2A \xleftarrow{d_3}
     \bigoplus_{u_1,u_2,u_3\in U} Au_1\otimes u_1Au_2\otimes u_2Au_3\otimes u_3A
     \xleftarrow{d_4}\cdots
  \]
  we obtain $S_*(A;A(U);A)$. So, it is clear that
  \[ H_n(R^U_*) = \Tor^{A(U)}_{n-1}(A,A) \] for every $n\geq 2$.
  For lower degrees we must consider
  \[ A \xleftarrow{d_1} \bigoplus_{u\in U} Au\otimes uA\xleftarrow{d_2} 
     \bigoplus_{u_1,u_2\in U} Au_1\otimes u_1Au_2\otimes u_2A 
  \]
  Since $d_1$ is really the multiplication morphism, for $n=0$ we will
  get
  \[ H_0(R^U_*) = \frac{A}{\sum_{u\in U} AuA} \] Then $ker(d_1) =
  ker\left(\bigoplus_{u\in U} Au\otimes uA\to A\right)$ and $H_0
  S_*(A;A(U);A) = \sum_{u\in U} Au\otimes_{A(U)}uA$ we get 
  \[ H_1(R^U_*) = ker\left(\sum_{u\in U} Au\otimes_{A(U)} uA\to
    A\right) \] We have similar results for $R^U_*(X)$ and $R^U_*(Y)$,
  here written only for $X$ as follows:
  \[ H_n R^U_*(X)=
     \begin{cases}
       X/(\sum_{u\in U} XuA) & \text{ if } n=0\\
       ker(\sum_{u\in U} Xu\otimes_{A(U)} uA\to X) & \text{ if } n=1\\
       \Tor^{A(U)}_{n-1}(X,A) & \text{ if } n\geq 2
     \end{cases}
  \]
\end{definition}

\begin{definition}
  One can define a graded multiplication structure on $R^U_*$ as
  follows: for any $\mathbf{x}=(x_0\otimes\cdots\otimes x_n)\in R^U_n$
  and $\mathbf{y}=(y_0\otimes\cdots\otimes y_m)\in R^U_m$ we define
  $\mathbf{x}\cdot\mathbf{y}\in R^U_{n+m}$ as
  \begin{align*}
    \mathbf{x}\cdot\mathbf{y}:= (x_0\otimes\cdots\otimes
    x_{n-1}\otimes x_ny_0\otimes y_1\otimes\cdots\otimes y_m)
  \end{align*}
  Note that for every $(a_0u\otimes ua_1)\in R^U_1$ we have
  \[ (a_0u\otimes ua_1) = (a_0\otimes u)a_1 = a_0(u\otimes ua_1) \]
  and for $n\geq 2$ and $(a_0u_1\otimes u_1a_1u_2\otimes\cdots\otimes
  u_{n-1}a_{n-1}u_n\otimes u_na_n)\in R^U_n$ we see
  \begin{align*}
    (a_0u_1\otimes u_1a_1u_2\otimes\cdots\otimes
    u_{n-1}a_{n-1}u_n\otimes u_na_n)
    = (a_0u_1\otimes u_1)\cdots(a_{n-1}u_n\otimes u_n)a_n
  \end{align*}
  In other words, $R^U_*$ is generated (not necessarily freely) by
  elements of degree $1$.
\end{definition}

\begin{proposition}
  $R^U_*$ is a differential graded $k$-algebra. Moreover, if $X$ is a
  right $A$-module then $R^U_*(X)$ is a differential graded right
  $R^U_*$-module. A similar statement holds for a left $A$-module $Y$
  and $R^U_*(Y)$.
\end{proposition}

\begin{proof}
  Since $R^U_*$ is generated by elements of degree 1, we must check
  whether the Leibniz rule is satisfied only for elements of degree 0
  and generators of degree 1. We have
  \[ d((a_0u\otimes u)a_1) = a_0ua_1 = d(a_0u\otimes u)a_1 \] 
  and
  \[ d(a_0(u\otimes ua_1)) = a_0ua_1 = a_0d(u\otimes ua_1) \] for any
  $(a_0u\otimes ua_1)\in R^U_1$. For two generators of degree $1$ we
  check
  \begin{align*}
    d((au\otimes u)(bv\otimes v))
    = & d(au\otimes ubv\otimes v)
    = (aubv\otimes v) - (au\otimes ubv)
  \end{align*}
  On the other hand
  \begin{align*}
    d(au\otimes u)(bv\otimes v) - (au\otimes u)d(bv\otimes v)
    = & (aubv\otimes v) - (au\otimes ubv)
  \end{align*}
  are equal. So, the Leibniz rule
  \[ d(\mathbf{x}\cdot\mathbf{y}) = d(\mathbf{x})\cdot\mathbf{y} +
  (-1)^{|\deg(\mathbf{x})|}\mathbf{x}\cdot d(\mathbf{y})
  \]
  holds.
\end{proof}

\begin{proposition}\label{Morita}
  If $H_0(R^U_*) = A/(\sum_{u\in U} AuA)=0$ then $A$ is a flat
  $A(U)$-module and $H_n(R^U_*(X))=0$ for every $A$-module $X$ and for
  every $n\geq 0$. In that case $A$ is Morita equivalent to $A(U)$,
  and therefore we get $\gldim(A)=\gldim(A(U))$.
\end{proposition}

\begin{proof}
  We have $H_0(R^U_*)=0$ if and only if $A=\sum_{u\in U} AuA$.  In
  particular, there exists $\alpha_u,\beta_u\in A$ such that $1=
  \sum_{u\in U} \alpha_u u\beta_u$.  Since $1= \sum_{u\in U} \alpha_u
  u\beta_u$ and we see that the action morphism $XU\otimes_{A(U)}
  UA\to X$ is surjective for every right $A$-module $X$,
  i.e. $H_0(R^U_*(X))=0$. Let $\sum_{u\in U} x_u u \otimes u a_u$ be
  in the kernel of the same morphism, i.e. let $0=\sum_{u\in U} x_u u
  a_u$. Then
  \begin{align*}
    \sum_u x_u u\otimes u a_u = & \sum_{u,u'} x_u u \otimes
    u a_u \alpha_{u'} u' \beta_{u'} = \sum_{u,u'} x_u u a_u
    \alpha_{u'} u' \otimes u'\beta_{u'} = 0
  \end{align*}
  since the tensor product is over $A(U)$. In other words, the action
  morphism is also injective, i.e. $H_1(R^U_*(X))=0$.  Put $P = AU :=
  \sum_{u\in U} Au$ and $Q=UA:= \sum_{u\in U} uA$ and then we see that
  $Q\otimes_A P \cong A(U)$.  Then $P\otimes_{A(U)}Q$ is isomorphic to
  $A$ as a $A$-bimodule via the multiplication morphism.  The result
  follows that $A$ is Morita equivalent to $A(U)$.  This means $UA$
  and $AU$ are projective $A(U)$-modules. Then $A$ is a flat
  $A(U)$-module. This gives us
  $H_n(R^U_*(X))=\Tor^{A(U)}_{n-1}(X,A)=0$ for $n\geq 1$.
\end{proof}

\begin{remark}
  The result we obtain in Proposition~\ref{Morita} shows obvious
  similarities with \cite{Palmer:SemitrivialExtensions} for artinian
  algebras if one considers Remark~2, Theorem~2 and Remark~3 of {\em
    ibid.} together. However, our proof is specific to artinian
  algebras, and there are differences in assumptions such as using
  flatness as opposed to projectivity.  Also, we realize the quotient
  $A/(\sum_{u\in U} AuA)$ as the 0-th homology of a differential
  graded algebra which acts on the (co)homology groups in a crucial
  way, as opposed to a separate invariant.
\end{remark}

\subsection{Flatness of $A$ as a bimodule over its subalgebras}
	
\begin{lemma}\label{FlatLowerBound}
  Assume that there exists a subset $U\subset E$ such that $A$ viewed
  as a (bi-)module over $A(U)$ is flat. Then $\gldim(A(U)) \leq
  \gldim(A)$
\end{lemma}

This lemma appears as Lemma 1 of
\cite{McConnell:GlobalDimensionSomeRings}.  The proof we present below
is different.
\begin{proof}
  Assume $M$ is a left $A(U)$-module. Then the induced module
  $Ind_{A(U)}^A(M) := \bigoplus_{u\in U} Au\otimes_{A(U)} uM$ is a
  left $A$-module. Let $P_*$ be a $A$-projective resolution of the
  induced $A$-module $Ind_{A(U)}^A(M)$ of length
  $\prdim_{A}(Ind_{A(U)}^A(M))\leq \gldim(A)$. Since $\bigoplus_{u\in
    U} uA$ is right $A$-projective and left $A(U)$-flat, and since the
  functor $\bigoplus_{u\in U} uA\otimes_A(\ \cdot\ )$ sends projective
  $A$-modules to flat $A(U)$-modules, we see that
  \[ \bigoplus_{u\in U} u P_* = \bigoplus_{u\in U} uA\otimes_A P_* \]
  is right $A(U)$-flat resolution of
  \[ M = \bigoplus_{u\in U}\sum_{u'\in U} uA\otimes_A Au' \otimes_{A(U)} u'M\]
  Then we easily see that $\wgldim(A(U))\leq \gldim(A)$. The result
  follows since $A(U)$ is also artinian.
\end{proof}

\begin{definition}
  Assume we have a directed graph $G$, and a 2-coloring of vertices,
  i.e. the set of vertices is split as a union of two disjoint
  subsets. In this set-up an oriented path $\alpha$ is called {\em
    alternating} if any two consecutive vertices have opposite colors.
\end{definition}

\begin{proposition}\label{Smear}
  Assume there exists a proper subset $U\subset E$ with the property
  that (i) $U$ is not a source or a sink in $\C{G}_A/U$, (ii)
  $H_0(R^U_*):= A/(\sum_{u\in U} AuA)$ is non-zero, (iii) $A$ is flat
  over both $A(U)$ and $A(U^c)$, and (iv) the global dimension $\gldim
  A(U)$ and $\gldim A(U^c)$ are both finite. In that case $\gldim(A)$
  is infinite if and only if $\C{G}_A$ has at least one alternating
  cycle with respect to the coloring $E = U \cup U^c$.
\end{proposition}

\begin{proof}
  Assume $X=Y=H_0(R^U_*):=A/\left(\sum_{u\in U} AuA\right)$ is
  non-trivial. Consider $S_*(X;A;Y)$ and the filtration we defined in
  the proof of Theorem~\ref{SourceSink}.  Note that since $U$ is
  neither a source nor a sink in $\C{G}_A/U$, we see that $UAU^c$ and
  $U^cAU$ are both non-zero.  The first column of the associated
  spectral sequence in the $E^1$-page yields
  $E^1_{0,q}=\Tor^{A(U^c)}_q(X,Y)$ and in particular
  \[ E^1_{0,0} = X U^c\otimes_{A(U^c)} U^cY \] This is because $X=Y$
  are $A$-modules with the property that action of the subalgebra
  $A(U)$ is given by 0. We also see that $E^1_{1,0}$ is
  \[ XU\otimes_{A(U)} UAU^c\otimes_{A(U^c)} U^cY \oplus
  XU^c\otimes_{A(U^c)} U^cAU\otimes_{A(U)} UY
  \]
  and it is trivial.  Moreover, since we assumed $A$ is flat over
  $A(U)$ and $A(U^c)$, the spectral sequence collapses on two non-zero
  axis.  Since $U$ is not a source or a sink in $\C{G}_A/U$, one can
  write paths of arbitrarily large lengths alternating between $U$ and
  $U^c$ if and only if $\C{G}_A$ has at least one alternating path.
  This means on the $q=0$ row of the the $E^1$-page all odd
  dimensional vector spaces $E^1_{2m+1,0}$ are zero and there are
  non-trivial term of arbitrarily large even dimension if and only if
  $\C{G}_A$ has at least one alternating path. The phenomenon is
  repeated in homology since the differentials are necessarily 0.
  This means in the $E^2$-page on $q=0$ row we have non-zero terms
  with arbitrarily large degrees if and only if $\C{G}_A$ has at least
  one alternating path. The result follows because the $p=0$ column
  has bounded height due to the fact that $A(U^c)$ has finite global
  dimension.
\end{proof}

\subsection{The algorithm}\label{Explanation}~

\begin{figure}[h]
  \centering\small
  \tikzstyle{test} = [rectangle, draw, fill=blue!15, inner sep=3pt, 
                      text width=2.6cm, text badly centered, 
                      minimum height=1.2cm]
  \tikzstyle{conclusion} = [ellipse, draw, fill=green!15, inner sep=1pt, 
                              text width=1.9cm, text badly centered]
  \tikzstyle{line} = [draw, -latex']
  \begin{tikzpicture}[auto,>=latex', thick]
    \draw node[conclusion] (Start)
    { Start with $U\subsetneq E$ };
    \draw node[test, below of=Start, node distance=2cm] (SourceSink)
    { $U$ is a source or a sink in $\C{G}_A/U$ };
    \draw node[test, right of=SourceSink, node distance=4cm] (Morita)
    { $A = \sum_{u\in U} AuA$ };
    \draw node[conclusion, below of=Morita, node distance=2cm] (Reduction)
    { $\gldim A= \gldim A(U)$ };
    \draw node[test, below of=SourceSink, node distance=2cm] (TestFinite)
    { $\gldim A(U)$ and $\gldim A(U^c)$ are both finite. };
    \draw node[conclusion, below of=TestFinite, node distance=2.3cm] (Finite)
    { $\gldim A$ is finite. };
    \draw node[test, right of=Morita, node distance=4cm] (TestFlat1)
    { $A$ is flat over $A(U)$ };
    \draw node[test, below of=TestFlat1, node distance=2cm] (TestFinite2)
    { $\gldim A(U)$ is finite. };
    \draw node[test, below of=TestFinite2, node distance=2cm] (TestFlat2)
    { $A$ is flat over $A(U^c)$ };
    \draw node[test, below of=TestFlat2, node distance=2.6cm] (TestFinite3)
    { $\gldim A(U^c)$ is finite. };
    \draw node[conclusion, below of=Reduction, node distance=2cm] (Infinite)
    { $\gldim A$ is infinite.};     
    \draw node[test, below of=Finite, node distance=2.3cm] (Alternating)
    { $\C{G}_A$ has a $U$-$U^c$ alternating cycle. };
    \draw node[conclusion, right of=TestFlat1, node distance=4cm, fill=red!15] (Stop)
    { Test another $U\subsetneq E$. };
    \path[line] (Start) -> (SourceSink);
    \path[->] (SourceSink) edge node {YES} (TestFinite);
    \path[->] (SourceSink) edge node {NO} (Morita);
    \path[->] (Morita) edge node {NO} (TestFlat1);
    \path[->] (Morita) edge node {YES} (Reduction);
    \path[->] (TestFinite) edge node {YES} (Finite);
    \path[->] (TestFinite) edge node {NO} (Infinite);  
    \path[->] (TestFinite2) edge node {YES} (TestFlat2);
    \path[line] (TestFlat2) -- +(2,0) -| node[near end] {NO} (Stop);
    \path[<-] (Infinite) edge node {YES} (Alternating);
    \path[<-] (Finite) edge node {NO} (Alternating);
    \path[<-] (Infinite) edge node {NO} (TestFinite2);     
    \path[->] (TestFlat1) edge node {YES} (TestFinite2);     
    \path[->] (TestFlat2) edge node {YES} (TestFinite3);     
    \path[->] (TestFinite3) edge node {NO} (Infinite);     
    \path[->] (TestFinite3) edge node {YES} (Alternating);     
    \path[->] (TestFlat1) edge node {NO} (Stop);       
  \end{tikzpicture}    
  \normalsize
  \caption{An algorithm to estimate the global dimension of
    $A$.}\label{Algorithm}
\end{figure}
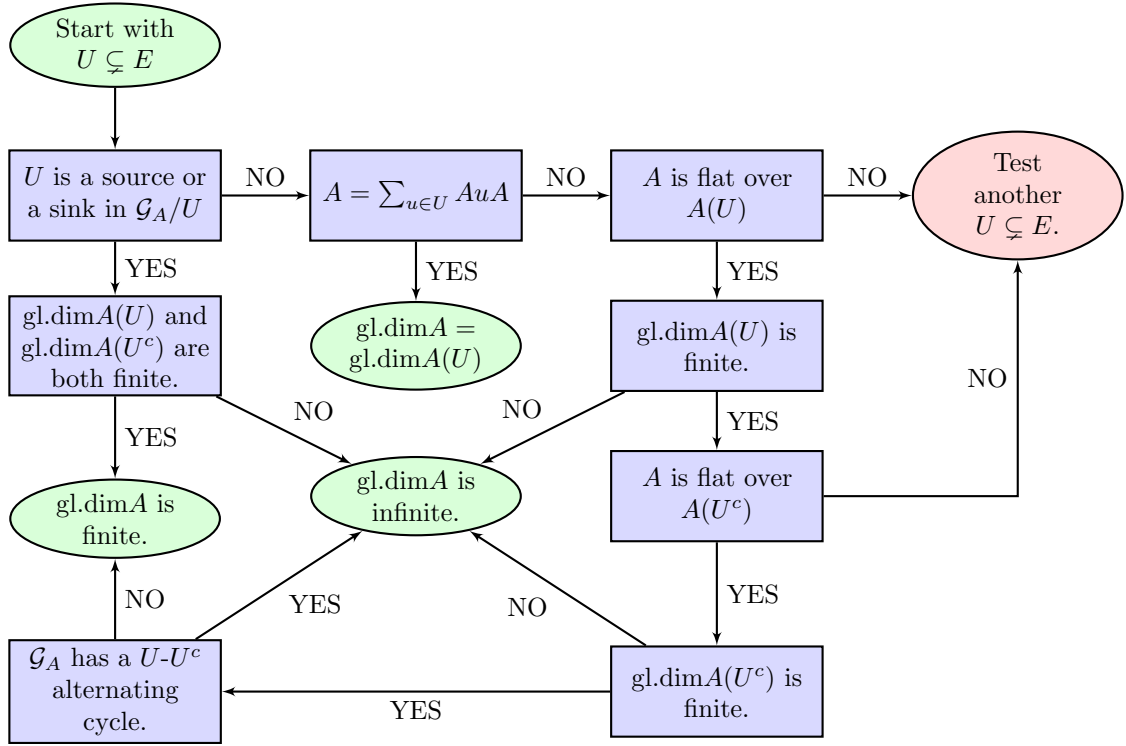

We summarize results we proved in this section in
Figure~\ref{Algorithm} as an algorithm. This algorithm allows us to
decide whether the global dimension of an artinian algebra is finite
or infinite for a large class of algebras.  However, there are also
cases in which our algorithm fails to produce a definitive answer. We
start with a proper subset $U\subsetneq E$.  If $U$ we chose is a
source or a sink in the quotient graph $\C{G}_A/U$, we use
Theorem~\ref{SourceSink} and reduce the calculation of the global
dimension of $A$ to computing the global dimensions of the subalgebras
$A(U)$ and $A(U^c)$.  Then we apply the algorithm to the subalgebras
$A(U)$ and $A(U^c)$ recursively. If the subset $U$ we chose satisfies
the condition that $A = \sum_{u\in U} AuA$ then we observe that $A$
and $A(U)$ have the same global dimension by
Proposition~\ref{Morita}. In that case we replace $A$ by $A(U)$ and
proceed recursively.  If it happens that that $U$ is not a source or a
sink in the quotient graph $\C{G}_A/U$, and that $A$ is not equal
$\sum_{u\in U}AuA$ we use a combination of Lemma~\ref{FlatLowerBound}
and Proposition~\ref{Smear} to proceed.  The algorithm will reduce the
calculation of upper and lower bounds for global dimensions of an
artinian algebra $A$ at hand until it can no longer reduce $A$ using
the graph $\C{G}_A$.  The type of artinian algebras that can not be
further reduced is called {\em non-recursive} whose definition we give
below.  The algorithm we described step by step in this section will
determine the finiteness of global dimension of $A$ in terms of the
global dimensions of some subset (which is not unique and will depend
on the choices made) of possibly non-recursive artinian subalgebras.

\begin{definition}\label{DefnRecursive}
  An artinian algebra $A$ is called {\em non-recursive} if it
  satisfies the following conditions {\bf for every} $U\subsetneq E$:
  \begin{enumerate}
  \item $U$ is not a source or a sink in $\C{G}_A/U$ 
  \item $A\neq \sum_{u\in U} AuA$ 
  \item $A$ is not flat over $A(U)$, or $A$ is not flat over $A(U^c)$
    and $\gldim A(U)$ is finite.
  \end{enumerate}
  Otherwise, we will call $A$ as {\em recursive}. 
\end{definition}

\section{Incidence algebras and Quotients of free path
  algebras}\label{PathAlgebras}

\subsection{Definitions and basic properties}~

Assume $\C{Q}$ is a finite directed graph. The path algebra $k\C{Q}$
of the finite directed graph $\C{Q}$ is the vector space spanned by
all the paths in $\C{Q}$. The multiplication on the paths is given by
the concatenation of the sequences of edges of the paths $p$ and $q$
if source of $p$ is the same as the range of $q$ and zero
otherwise. The set of all vertices in $\C{Q}$ form a complete set of
orthogonal idempotents.  Moreover, each vertex in $\C{Q}$, is a
primitive idempotent of the path algebra, and vice versa. Hence, the
directed graph $\C{G}_{k\C{Q}}$ of the path algebra $k\C{Q}$ in the
sense of Definition~\ref{GraphOfAnArtinianAlgebra} has the same vertex
set as of $\C{Q}$. For any non-zero path $\alpha$ from vertex $e$ to
vertex $f$, $\alpha = f \alpha e $ is in $f(k\C{Q})e$ which is a
non-zero $k$-vector space of $k\C{Q}$. So there is an edge connecting
$e \rightarrow f $ in $\C{G}_{k\C{Q}}$ and the directed graph
$\C{G}_{k\C{Q}}$ is an extension of the directed graph $\C{Q}$. If
$\C{Q}$ does not have any oriented cycles, neither does
$\C{G}_{k\C{Q}}$.

\begin{lemma}\label{FlatLemma}
  Let $A:=k\C{Q}$ be the free path algebra over a finite directed
  graph $\C{Q}$. Then $A$ is flat over $A(U)$ for every subset $U$ of
  vertices in $\C{Q}$.
\end{lemma}

\begin{proof}
  Let $U$ be a subset of vertices in $\C{Q}$.  Then $UA:= \sum_{u\in
    U} uA$, the left ideal of paths ending at a vertex in $U$, has a
  basis over $A(U)$.  This basis is the set of all paths starting at a
  vertex in $U^c$, ending at a vertex in $U$ and consist of no cycles
  within $U$.  It is clear that this is a generating set, let us show
  that it is a basis.  Assume we have distinct paths
  $\beta_1,\ldots,\beta_n$ in this generating set such that there are
  elements $\alpha_1,\ldots, \alpha_n$ in $A(U)$ with the property
  that $\sum \alpha_i \beta_i =0$.  Then for every path $\gamma$
  appearing in $\alpha_1$ there is a path $\beta_j$ and a path
  $\gamma'$ in $A(U)$ appearing in $ \alpha_j$ such that $\gamma
  \beta_1 = \gamma' \beta_j$. Without loss of generality, assume the
  length of $\beta_1$ is less than or equal to the length of
  $\beta_j$. Then, $\beta_1$ is a part of $\beta_j$, which forces
  $\beta_j$ to go through the vertex set $U$ at least twice. This is a
  contradiction.  Then the lengths of $\beta_1$ and $\beta_j$ are the
  same, and therefore, they are the same. This is also a
  contradiction. So, we conclude that the elements
  $\beta_1,\ldots,\beta_n$ are linearly independent which means $UA$
  is free over $A(U)$, and therefore flat. Now, $A$ splits as a direct
  sum $A= UA \oplus U^cA$. The action of $A(U)$ on $U^cA$ is by zero
  making it flat over $A(U)$.  Therefore, $A$ is flat over $A(U)$.
  The proof for $A$ viewed as a right module $A(U)$ is similar.
\end{proof}

\begin{proposition}\label{FreePathAlgebraIsHereditary}\cite[Proposition 1.4]{AuslanderReitenSmalo:RepresentationTheoryofArtinAlgebras}
  Assume $\C{Q}$ is a finite directed graph with no oriented
  cycles. Then the free path algebra $k\C{Q}$ of $\C{Q}$ is a
  hereditary $k$-algebra.
\end{proposition}

\begin{proof}
  Since $k\C{Q}$ is a finite dimensional $k$-algebra, it is
  artinian. Moreover, $\C{Q} \subset \C{G}_{k\C{Q}}$ has no oriented
  cycles and $e(k\C{Q})e = k$ for every vertex in $\C{Q}$ is of global
  dimension 0.  The result follows using Corollary~\ref{NoCycle}.
\end{proof}

\begin{theorem}\label{QuotientPathAlgebra}
  Let $A$ be the free path algebra of a cycle-free directed graph
  $\C{Q}=(\C{Q}_0,\C{Q}_1)$ and $I$ be a nil ideal of $A$. Then the
  artinian algebra $B:=A/I$ has finite global dimension.
\end{theorem}

\begin{proof}
  Let $\pi:A \rightarrow B$ be the canonical quotient morphism.  Let
  $E_B$ denote a set of complete primitive idempotents of the artinian
  algebra $B$. Since $I$ is nil, for each $u\in E_B$ we can pick a
  primitive idempotent $\hat{u}\in A$ such that $\pi(\hat{u})=u$ by
  Lemma~\ref{IdempotentLifting} and
  Lemma~\ref{PrimitiveIdempotentLifting}.  Let $e = \sum_{u\in E_B}
  \hat{u}$ the sum of all of these primitive idempotents, and let $f =
  1_A - e$.  Then $fAf$ is an artinian algebra with unit $f$, and
  therefore, has its own set of primitive idempotents splitting its
  unit $f$.  Thus we completed the lifted set of primitive idempotents
  $\{\hat{u}\in A|\ u\in E_B\}$ to a full set of primitive idempotents
  $E_A$ splitting $1_A$.  This set of primitive idempotents might not
  be necessarily the set $\C{Q}_0$, but one can find a bijection as in
  Proposition~\ref{SplittingIdempotents} and the resulting graph is
  isomorphic to $\C{G}_A$ by Lemma~\ref{UniqueGraph}.  Now, assume by
  way of contradiction that the directed graph $\C{G}_B$ of $B$ has a
  cycle $\alpha$ which starts and ends in a primitive idempotent
  $u$. Then there exists a lift $\widehat{\alpha}\in A$ of the cycle
  $\alpha$ which is a linear combination of cycles which start and end
  in an idempotent $\hat{u}\in E_A$. This is a contradiction since
  $\C{Q}$, and therefore $\C{G}_{A}$, has no cycles. Now, by
  Corollary~\ref{NoCycle} we get the finite upper bound $|E|-1$
  because $eBe \cong k$ is of global dimension 0 for every $e\in E$.
\end{proof}

\begin{remark}
  Note that if $I$ is a nil ideal then it contains no idempotents other 
  than 0.  In particular, it contains no elements from the set of 
  primitive idempotents.
\end{remark}

For any partially ordered set $X$, define the corresponding incidence
algebra $I(X)$ over $k$ as the set of all functions $f:X \times X
\rightarrow k$ with $f(x,y)=0$ unless $x \leq y$, together with the
operations
\begin{align*}
  (f+g)(x,y) = & f(x,y)+g(x,y)\\
  fg(x,y) = & \sum_{x \leq z \leq y}f(x,z)g(z,y)\\
  (rf)(x,y) = & rf(x,y)  
\end{align*}
for any $r \in k$ and $f,g \in I(X)$ and $x,y \in X$.  However, since
the partially ordered sets we consider are all finite, we can identify
our incidence algebras as algebras generated by symbols $E_{xy}$ for
every $x\leq y$ subject to the relations
\begin{enumerate}
  
\item $E_{xt}E_{yz}=0$ for $x,y,z,t \in X$ when $x \leq t$ and $y \leq
  z$, and finally $t \neq y$.
  
\item $E_{xz}=E_{xy}E_{yz}$ for $x,y,z \in X$ when $x \leq y \leq z$.
  
\item $E_{xy}=E_{xx}E_{xy}=E_{xy}E_{yy} $ for $x,y \in X$ when $x \leq
  y $.
  
\end{enumerate}

A complete set of primitive idempotents of $I(X)$ splitting the
identity is $\{E_{xx} : \, x \in X\}$.  Hence, the set of all vertices
of $\C{G}_{I(X)}$ is exactly $X$ and there is an edge of the form $x
\leftarrow y$ in $\C{G}_{I(X)}$ if and only if $x\leq y$ in $X$.

\begin{corollary}\label{IncidenceAlgebras}\cite[Section 9]{Mitchell:TheoryOfCategories}
  The incidence algebra of any finite poset has finite global
  dimension.
\end{corollary}

\begin{proof} 
  Any finite poset $X$ determines a directed graph $G_X=(X,E)$ where
  edges are defined by elements which are in relation.  Then the
  incidence algebra $I(X)$ over $X$ is isomorphic to a quotient of the
  free path algebra $kG_X$ of this graph as follows.  One can define
  an algebra epimorphism by mapping each path $\alpha$ in $kX$ to the
  edge $E_{xy}$ which is an element in $I(X)$ where $x$ is the
  terminal point (target) of the path $\alpha$ and $y$ is the initial
  point (source) of $\alpha$. The kernel of this epimorphism is the
  subalgebra generated by the difference of paths with the same source
  and target. Then the result follows from
  Theorem~\ref{QuotientPathAlgebra}.
\end{proof}

\subsection{Calculations}\label{Examples}

\begin{example}
  Consider the directed graph $\C{Q}$
  \[\xymatrix{
    e \bullet \ar@/{}^{.5pc}/[r]^x & \ar@/{}^{0.5pc}/[l]^y \bullet f
  }\]
  and the free path algebra $A:=k\C{Q}$ generated by this graph over
  our base field $k$.  Let $U=\{e\}$ and $U^c=\{f\}$.  We see that
  $A(U) = \text{Span}_k(e,yx,(yx)^2,(yx)^3,\ldots)$ and $A(U^c) =
  \text{Span}_k(f,xy,(xy)^2,(xy)^3,\ldots)$ are isomorphic to
  polynomial algebras over one indeterminate over $k$ which are
  commutative and of global dimension 1.  Moreover, $A$ splits as an
  $A(U)$ module
  \[ A = A(U) \oplus x A(U) \oplus A(U) y \oplus x A(U) y \oplus
  \text{Span}_k(f)
  \] and as an $A(U^c)$-module
  \[ A(U^c) \oplus y A(U^c) \oplus A(U^c) x \oplus y A(U^c) x \oplus
  \text{Span}_k(e)
  \] This means $A$ is flat over $A(U)$ and $A(U^c)$.  Moreover, we
  have $H_0(R^U_*) \cong H_0(R^{U^c}_*)\cong k$ both are
  non-zero. Then by Proposition~\ref{Smear} we see that
  $\gldim(k\C{Q})$ is infinite.  
\end{example}

\begin{example}\label{NilpotentCycleExample}
  Consider the same directed graph $\C{Q}$ and we define $A :=
  k\C{Q}/\left<xy,yx\right>$.  Then $A$ is a finite dimensional
  algebra $A = \text{Span}_k(e,f,x,y)$.  Let $U=\{e\}$ and
  $U^c=\{f\}$.  We see that $A(U) \cong k \cong A(U^c)$ which are
  semi-simple, and therefore, of global dimension 0. This means $A$ is
  flat over $A(U)$ and $A(U^c)$.  Moreover, we have $H_0(R^U_*) \cong
  H_0(R^{U^c}_*)\cong k$ both are non-zero. Then by
  Proposition~\ref{Smear} we see that $\gldim(A)$ is infinite.
\end{example}

\begin{example}
  Consider the same directed graph $\C{Q}$ and let
    $A:=k\C{Q}/\left<xy\right>$.  We get a finite dimensional algebra
    \[ A = \text{Span}_k(e,f,x,y,yx) \] Again, let $U=\{e\}$ and
    $U^c=\{f\}$.  We immediately see that $A(U^c) = \text{Span}_k(f)
    \cong k$ is of global dimension 0.  On the other hand, $A(U)
    =\text{Span}_k(e,yx)$ of infinite global dimension because it is
    isomorphic to $k[t]/\left<t^2\right>$, and $A$ splits as a
    $A(U)$-bimodule as
    \[ A = A(U) \oplus \text{Span}_k(f,x,y) \] Then $A$ is flat over
    $A(U)$, and therefore, by Lemma~\ref{FlatLowerBound} we conclude
    that $A = k\C{Q}/\left<xy\right>$ is of infinite global
    dimension.
\end{example}

\begin{example}
  Notice that Example 3.6 in \cite{AnickGreen:QuotientOfPathAlgebras}
  considers the path algebra $A$ over $\C{Q}$
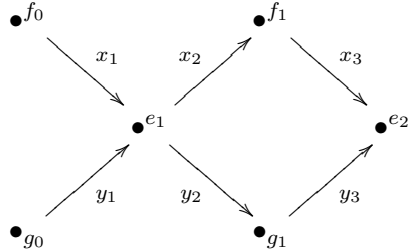
  \[ \xymatrix{&{\bullet}^{f_0}\ar@{->}^{x_1}[dr]&
    &{\bullet}^{f_1}\ar@{<-}_{x_2}[dl]\ar@{->}^{x_3}[dr]& \\
    &&{\bullet}^{e_1} &&{\bullet}^{e_2}&  \\
    &{\bullet}_{g_0}\ar@{->}_{y_1}[ur]&
    &{\bullet}_{g_1}\ar@{<-}^{y_2}[ul]\ar@{->}_{y_3}[ur] & } \]
  subject to the relations $x_2x_1=0=y_2y_1$, $x_3x_2=y_3y_2$ and
  computes the $\gldim(A) =2$. Here we use category theoretic
  composition convention when we multiply elements. With the procedure
  defined above we get $\C{G}_{A}$ as:
   
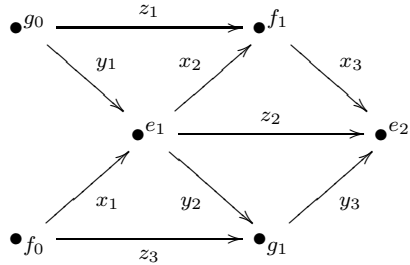
\[ \xymatrix{&{\bullet}^{g_0}\ar@{->}^{y_1}[dr]\ar[rr] ^{z_1} &  
     &{\bullet}^{f_1}\ar@{<-}_{x_2}[dl]\ar@{->}^{x_3}[dr]& \\
             &&{\bullet}^{e_1}\ar[rr]^{z_2} & & {\bullet}^{e_2} & \\
             &{\bullet}_{f_0}\ar@{->}_{x_1}[ur]\ar[rr]_{z_3} &
             &{\bullet}_{g_1}\ar@{<-}^{y_2}[ul]\ar@{->}_{y_3}[ur]
             &
   } \]
   Let $U=\{f_0,g_0 \}$, then $A(U)= f_0Af_0 \oplus g_0Ag_0$ which is
   isomorphic to $k \oplus k$, So $\gldim(A(U))=0$. Let $B=A(U^c)$.
   Now, pick new $U=\{e_1,f_1,g_1 \}$, then $B(U)$ is a free path
   algebra that is not semi-simple and $B(U^c)=e_2Be_2 \cong k$.
   Hence, $\gldim(A) \leq 3$ is an upper bound which does not
   contradict the actual dimension.  
 \end{example}
 
\begin{example}
Consider the directed graph $\C{Q}$, 
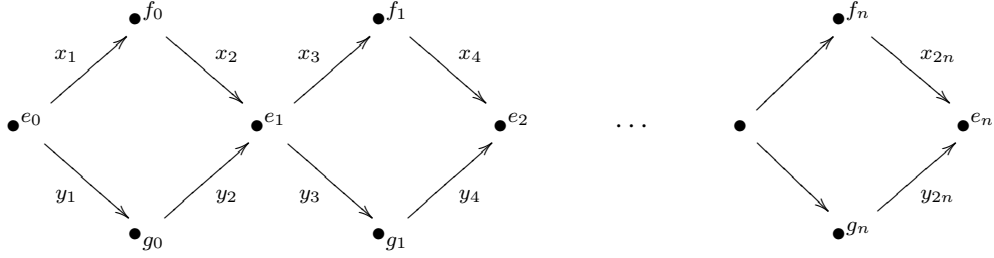
\[ \xymatrix{
&{\bullet}^{f_0}\ar@{<-}_{x_1}[dl]\ar@{->}^{x_2}[dr]& 
&{\bullet}^{f_1}\ar@{<-}_{x_3}[dl]\ar@{->}^{x_4}[dr]& & &&
{\bullet}^{f_n}\ar@{<-}[dl]\ar@{->}^{x_{2n}}[dr]\\
             {\bullet}^{e_0}&&{\bullet}^{e_1} &&{\bullet}^{e_2}& \cdots 
             & {\bullet} & & {\bullet}^{e_n}\\
             &{\bullet}_{g_0}\ar@{<-}^{y_1}[ul]\ar@{->}_{y_2}[ur]&
             &{\bullet}_{g_1}\ar@{<-}^{y_3}[ul]\ar@{->}_{y_4}[ur]
             & & & & 
{\bullet}^{g_n}\ar@{<-}[ul]\ar@{->}_{y_{2n}}[ur]
} \]
and let $A$ be the path algebra over $\C{Q}$ subject to the relations 
\[ x_{2k}x_{2k-1}=y_{2k}y_{2k-1}\qquad
x_{2k+1}x_{2k}=0=y_{2k+1}y_{2k} \] for $k=1,2,...,n$ with the same
composition convention as above.

Then $\C{G}_{A}$ will be:
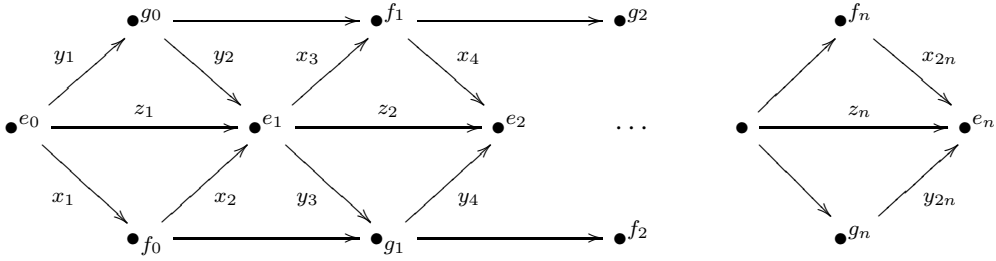
\[ \xymatrix{&{\bullet}^{g_0}\ar@{<-}_{y_1}[dl]\ar@{->}^{y_2}[dr]
\ar[rr] & & {\bullet}^{f_1}\ar@{<-}_{x_3}[dl]\ar@{->}^{x_4}[dr]\ar[rr]
& & {\bullet}^{g_2}&&
{\bullet}^{f_n}\ar@{<-}[dl]\ar@{->}^{x_{2n}}[dr]\\
             {\bullet}^{e_0}\ar[rr]^{z_1}& &
             {\bullet}^{e_1}\ar[rr] ^{z_2}& & {\bullet}^{e_2}& \cdots 
             & {\bullet} \ar[rr]^{z_n} & & {\bullet}^{e_n}\\
             &{\bullet}_{f_0}\ar@{<-}^{x_1}[ul]\ar@{->}_{x_2}[ur]\ar[rr] & & 
             {\bullet}_{g_1}\ar@{<-}^{y_3}[ul]\ar@{->}_{y_4}[ur] 
             \ar[rr]& & {\bullet}^{f_2}&& 
{\bullet}^{g_n}\ar@{<-}[ul]\ar@{->}_{y_{2n}}[ur]
} \]
Let $U=\{e_0,f_0,g_0 \}$, then 
$A(U)=e_0Ae_0 \oplus e_0Af_0 \oplus e_0Ag_0 \oplus f_0Af_0 \oplus g_0Ag_0$.
Since $A(U)$ is a free path algebra, it is hereditary. Moreover, $A(U)$ is 
algebra isomorphic to 
$\left [ \begin{array}{ccc} k & k& k \\ 0&k&0 \\ 0&0&k\end{array} \right ]$ 
which is not semi-simple, as the submodule 
$\left [ \begin{array}{ccc} k & k& k \\ 0&0&0 \\ 0&0&0\end{array} \right ] $ 
is not a direct summand. So $\gldim(A(U))=1$.
By Theorem~\ref{SourceSink},  
\[ \gldim(A) \leq 1 + \gldim(A(U)) + \gldim(A(U^c)) = 2 +
\gldim(A(U^c)) \] Notice that if we call $A=A_n$ and $\C{Q}=\C{Q}_n$,
then $A(U^c)$ is the path algebra on $\C{Q}_{n-1}$ subject to the
relations $x_{2k}x_{2k-1}=y_{2k}y_{2k-1}$,
$x_{2k+1}x_{2k}=0=y_{2k+1}y_{2k}$ for $k=2,3,...,n$. Use the same
procedure for the algebra $A_{n-1}$. Pick $U=\{e_1,f_1,g_1 \}$. We
get,
\[ \gldim(A_{n-1}) \leq 2 + \gldim(A_{n-2}). \] If we continue in a
similar manner, $A_0$ becomes $A_0=e_0Ae_0$ which is isomorphic to
$k$, so $\gldim(A_0) = 0$.  Hence, $ \gldim(A) \leq 2n $.
\end{example}
  

\end{document}